\documentclass{amsart}

\usepackage{todonotes}
\usepackage[english]{babel}
\usepackage{hyperref,amssymb,amsmath,amsthm,color,mathtools,tikz,subfig,dirtytalk,csquotes,soul,orcidlink}
\usetikzlibrary{arrows}

\numberwithin{equation}{section}

\newtheorem{example}{Example}[section]

\newtheorem{lemma}[example]{Lemma}
\newtheorem{remark}[example]{Remark}

\newtheorem{theo}[example]{Theorem}

\def\elle#1{L^{#1}}

\def\BV{BV((0,1),\{-1,1\})}

\def\e{\varepsilon}
\def\semin#1#2{ [#1 ]_{#2}}

\def\R{\mathbb{R}}

\newcommand{\abs}[1]{\left|#1\right|}

\DeclareRobustCommand{\rchi}{{\mathpalette\irchi\relax}}
\newcommand{\irchi}[2]{\raisebox{\depth}{$#1\chi$}} 

\title[Non-local phase transitions close to $s=\frac12$]{Analysis for non-local phase transitions close to the critical exponent $s=\frac12$}
\author{Marco Picerni
\orcidlink{0009-0004-4364-4831}}\email{mpicerni@sissa.it}
\address{SISSA, via Bonomea 265, 34136, Trieste, Italy}

\begin{document}

\maketitle
\section*{Abstract}
We analyze the behaviour of double-well energies perturbed by fractional Ga\-gliar\-do squared seminorms in $H^s$ close to the critical exponent $s=\frac12$. This is done by computing a scaling factor $\lambda(\varepsilon,s)$, continuous in both variables, such that
\[
\mathcal{F}^{s_\varepsilon}_\varepsilon(u)=\frac{\lambda(\varepsilon,s_\varepsilon)}{\varepsilon}\int W(u)\,dt+\lambda(\varepsilon,s_\varepsilon)\varepsilon^{(2s_\varepsilon-1)^+}\semin{u}{{s_\varepsilon}}^2
\]
$\Gamma$-converge, for any choice of $s_\varepsilon \to \frac12$ as $\e\to 0$, to the sharp-interface functional found by Alberti, Bouchitt\'e and Seppecher in \cite{ABS} with the scaling ${|\log\e|^{-1}}$. Moreover, we prove that all the values $s\in [\frac12,1 )$ are regular points for the functional $\mathcal{F}^{s}_\varepsilon$ in the sense of equivalence by $\Gamma$-convergence (see \cite{BT-GammaExp}), and that the $\Gamma$-limits as $\e\to 0$ are continuous with respect to $s$. In particular, the corresponding surface tensions, given by suitable non-local optimal-profile problems, are continuous on $[\frac12,1)$.

{\bf MSC codes:} 35B25, 35G20, 35R11, 49J45, 74A50, 82B26

{\bf Keywords:} Non-convex energies, fractional Sobolev spaces, $\Gamma$-convergence, phase transitions, Cahn-Hilliard functional, singular perturbations.

\section{Introduction and outline}\def\e{\varepsilon}
Singular perturbations are often used to select solutions of non-convex problems with multiplicity of minimizers. In the case of theories of phase transitions, usually the non-convex problem at hand is an integral depending on a scalar variable $u$ through a ``double-well potential''  $W$ (that is, a function with two minimizers $\alpha$ and $\beta$). In the classical Cahn--Hilliard theory of phase transitions \cite{CahnHilliard} such an energy is singularly perturbed by a term with the gradient of $u$ depending on a small parameter $\e$ as
\[
\int_\Omega W(u)\,dt+ \varepsilon^{2}\int_\Omega|\nabla u|^2dx.
\] As $\e\to 0$, minimizers under a volume constraint converge to a function taking only the values $\alpha$ and $\beta$, and such that the interface between these two phases is minimal. This minimal-interface criterion had been conjectured by Gurtin \cite{Gurtin} and proven by Modica \cite{Modica} using the $\Gamma$-convergence (cf. \cite{DalMasoBook}, \cite{BraidesBook}) results of a previous seminal paper by Modica and Mortola \cite{ModicaMortola}, obtaining that the functionals above, scaled by $\frac1\e$, $\Gamma$-converge to a functional defined on the space $BV(\Omega;\{\alpha,\beta\})$ of functions with bounded variation taking only the values $\alpha$ and $\beta$ by
\[
m_W \hbox{\rm Per}(\{u=\alpha\};\Omega),
\] 
where Per$(A;\Omega)$ denotes the perimeter of $A$ in $\Omega$ and $m_W$  (the {\em surface tension} between the phases) is a constant determined by $W$ only. Since functions in $BV(\Omega;\{\alpha,\beta\})$ can be identified with sets of finite perimeter $\{u=\alpha\}$, this result provides a proof of the minimal-interface criterion.
It is interesting that this result is one-dimensional in that $m_W$ is characterized by the problems
\[
m_W =\min\Bigl\{\int_{-\infty}^{+\infty} (W(v)+|\nabla v|^2)\,dt: v(-\infty)=\alpha,\  v(+\infty)=\beta
\Bigr\}
\] 
(see \cite{FonsecaTartar}). The result in \cite{Modica} has been extended in many ways, and in particular using higher-order gradients as in \cite{FonsecaMantegazza} and \cite{BDS-higherorder}, with analogous formulas for the corresponding surface tension, depending on the order of the perturbation.
 
Motivated by an interest in non-local problems, perturbations have also been taken in fractional Sobolev spaces,
with functionals of the form
\begin{equation}\label{F-e-s}
\int_\Omega W(u)\,dx+ \varepsilon^{2s}\semin{u}{{s}}^2
\end{equation}
where $\semin{u}{s}$ denotes the Gagliardo seminorm in the fractional Sobolev space $H^{s}(\Omega)$ with $s\in(0,1)$ (see the works of Alberti, Bouchitt\'e, and Seppecher \cite{ABS} \cite{ABS-full}, Savin and Valdinoci \cite{SVa-nonlocal-ext} and Palatucci-Vincini \cite{PVi-GammaLim-s01}).  If $s>\frac12$ then the scaling by $\frac1\e$ as in the Modica--Mortola case gives functionals
\[
\frac1\e\int_\Omega W(u)\,dx+ \varepsilon^{2s-1}\semin{u}{{s}}^2\,,
\]
which still provide a $\Gamma$-limit of perimeter type as above, with a surface tension $m_s$ depending on $s$. Recently, the result has been proved to hold also for higher-order fractional perturbations by Solci \cite{Solci-higherorder}, who also shows that the functionals can be slightly modified in such a way that the corresponding surface tension $m_s$ is a continuous function for $s\in(\frac12,+\infty)$. A modification is necessary by the critical behaviour of the fractional norms at integer points.

The case $s=\frac12$ is {\em critical}, in the sense that the scaling by $\frac1\e$ makes the coefficient in front of the $H^{\frac12}$-seminorm equal to $1$, so that, in order to obtain a phase-transition energy of perimeter type, in this case it is necessary to scale by a further logarithmic factor; that is, to consider functionals
\[
\frac1{\e|\log\e|}\int_\Omega W(u)\,dx+ \frac1{|\log\e|}\semin{u}{{\frac12}}^2\,.
\]
With this scaling, the $\Gamma$-limit is still an energy of perimeter type with the explicit surface tension $m_{\frac12}=8$. 

In the case $s<\frac12$, finally, Savin and Valdinoci have shown that the functionals scaled by $\e^{-2s}$ have a non-local phase-transition limit in which the domain are $H^s$ functions taking only the values $\alpha$ and $\beta$.

\smallskip

In this paper we aim at describing more in detail the behaviour of functionals \eqref{F-e-s} for $s$ {\em close} to $\frac12$. To that end, we introduce functionals
\begin{equation}\label{intro:funzionale}
    F_\varepsilon^s(u) = \frac{1}{\varepsilon} \int_{(0,1)} W(u) dt + \varepsilon^{(2s-1)^+} \iint_{(0,1) \times (0,1)} \frac{\abs{u(x)-u(y)}^2}{\abs{x-y}^{1+2s}} dx dy.
\end{equation}
Here, $W$ is a double-well potential satisfying the following assumptions:
\begin{itemize}
    \item $W$ is continuous;
    \item $W \geq 0$ and $W(z) = 0$ if and only if $z = \pm1$;
    \item $\liminf\limits_{z\to\pm\infty}W(z)>0$.
\end{itemize}
Using the approach by \cite{BT-GammaExp}, in order to provide a description as $s\to\frac12$ and $\e\to0$,  
we investigate the behaviour of the functionals $F_\varepsilon^s$ under the assumption that $s = s_\varepsilon \to \frac{1}{2}$ as $\varepsilon \to 0$. 
The question then translates in finding a scaling factor $\lambda(\varepsilon,s)$ depending continuously on its variables such that
\[
\lambda(\varepsilon,s_\varepsilon) F_\varepsilon^{s_\varepsilon} \xrightarrow{\Gamma} F_0^{\frac{1}{2}}
\]for all choices of $s_\e\to \frac12$.

Note that a particular case is $s_\varepsilon=\frac{1}{2}$ for all $\e$,
for which, in this one-dimensional case, the asymptotic result is that
\begin{equation}\label{def_F0_12}
    \Gamma\hbox{-}\lim_{\e\to 0} 
\frac{1}{\abs{\log(\varepsilon)}} F_\varepsilon^{\frac{1}{2}}= F_0^{\frac{1}{2}} =
\begin{cases}
    m_\frac12 \#S(u) & \text{if } u \in \BV, \\
    +\infty & \text{otherwise}.
\end{cases}
\end{equation}
with $m_\frac12=8$.
This remark implies that we may take  $\lambda(\varepsilon,\frac12)= \frac1{|\log\e|}$. Note that if the wells of $W$ are in two points $\alpha<\beta$, the surface tension $m_\frac12$ is equal to $2(\beta-\alpha)^2$. All the results we prove in this paper also hold in such generality.

\smallskip
In Sections \ref{sec:scaling} and \ref{sec:gammalim} we prove that such a scaling is 
\[
\lambda(\varepsilon,s)=\begin{cases}\frac{2s-1}{1-{\varepsilon}^{2s-1}} &\hbox{ if } s>\frac12\\
\frac1{|\log\e|} &\hbox{ if } s=\frac12\\
\frac{2s-1}{{\varepsilon}^{\frac{1-2s}{2s}}-1}&\hbox{ if } s<\frac12.
\end{cases}
\]
In terms of asymptotic behaviour, in particular, this implies that in the regime 
$$
\abs{2s_\e-1}<<\frac1{|\log\e|},
$$
we have a {\em separation of scales} effect; that is, the $\Gamma$-limit of $\lambda(\varepsilon,s) F_\varepsilon^{s}$ is the same as the one obtained by first letting $s\to \frac12$ with $\e>0$ fixed, which gives  $\frac1{|\log\e|} F_\varepsilon^{1/2}$, and then letting $\e\to 0$. 

Furthermore, in Section \ref{sec:ms} we show that this analysis extends to $s_\e\to s\in (\frac12,1)$ by studying the behaviour of the surface tensions 
\begin{equation}\label{intro:probms}
    m_s = \inf \left\{ \int_{\R} W(u) dt + \iint_{\R \times \R} \frac{\abs{u_\varepsilon(x)-u_\varepsilon(y)}^2}{\abs{x-y}^{1+2s}} dx dy \mid u \in H^s(\R), u(\pm \infty) = \pm 1 \right\},
\end{equation}
as $s\to\frac12$ and proving that 
\[\lim_{s\to\frac12^+}(2s-1)m_s=m_\frac12\,.\]
This condition gives a continuity of the description by the scaled functionals. Indeed, since for $s>\frac12$ and $s_\e\to s$ we have
\[
\lambda(\varepsilon,s_\e)=\frac{1-2s_\e}{{\varepsilon}^{2s_\e-1}-1}\to 2s-1
\]
as $\e\to 0$, we obtain that
\[ \Gamma\hbox{-}\lim_{\e\to 0} 
\lambda(\varepsilon,s_\varepsilon)
F_\varepsilon^{s_\varepsilon} =
\begin{cases}
    (2s-1)m_{s} \#S(u) & \text{if } u \in \BV, \\
    +\infty & \text{otherwise},
\end{cases}
\]
which tends to $F_0^{1/2}$ as $s\to\frac12^+$. Again, we highlight a separation of scales effect  if $s_\e\to\frac12^+$ and 
$$
2s_\e-1>>\frac1{|\log\e|}.
$$
In this case, the $\Gamma$-limit of $\lambda(\varepsilon,s) F_\varepsilon^{s}$ is the same as the one obtained first letting $\e\to 0$ with $s> \frac12$ fixed, which gives  $(2s-1)\frac{m_s}{m_{1/2}}F_0^{1/2}$, and then letting $s\to \frac12$. 

These results can be expressed in the terminology of $\Gamma$-expansions \cite{BT-GammaExp} 
as the equivalence of the functionals $F^s_\e$ and the functionals
\[
G^s_\e(u)=\begin{cases}(1-{\varepsilon}^{2s-1}) m_s F_0(u) & \text{if } \e>\frac12\\
|\log\e| m_{\frac12} F_0(u) & \text{if } \e=\frac12,\end{cases}
\]
where
\[
F_0 =
\begin{cases}
    \#S(u) & \text{if } u \in \BV \\
    +\infty & \text{otherwise},
\end{cases}
\]
for $s$ varying uniformly on compact sets of $[\frac12,1)$. We note that this result can be extended to all
compact subsets of  $[\frac12,+\infty)$ upon taking the correct extension to higher-order fractional perturbations as in \cite{Solci-higherorder}.

	\section{Finding the correct scaling factor}\label{sec:scaling}

        In this section, we consider a sequence $(s_\varepsilon)_\varepsilon$ such that
        $$\lim_{\varepsilon\to0}s_\varepsilon=\frac12.$$
        Our aim is to find a scaling factor $\lambda(\varepsilon)$ such that the functionals 
        \(\lambda(\varepsilon)F^{s_\varepsilon}_\varepsilon\),
        where $F^{s_\varepsilon}_\varepsilon$ are given by \eqref{intro:funzionale}, $\Gamma$-converge to $F_0^\frac12$ as $\varepsilon\to0$. If $s=\frac12$ then the result by Alberti, Bouchitt\'e and Seppecher gives $\lambda(\e)=\frac1{|\log\e|}$, so we can suppose $s_\e\neq \frac12$.
        
        Before beginning, we recall that the double-well potential $W$ satisfies the following assumptions:
            \begin{enumerate}
        \item\label{HP_W_1} $W$ is continuous;
        \item\label{HP_W_2} $W \geq 0$ and $W(z) = 0$ if and only if $z = \pm1$;
        \item\label{HP_W_3} $\liminf\limits_{z\to\pm\infty}W(z)>0$.
    \end{enumerate}

        Moreover, for \(s \in (0,1)\) and \(v \in H^s(0,1)\), we define the Gagliardo seminorm $\semin{v}{s}$ by
        \begin{equation}\label{def_Gagliardo}            
        \semin{v}{s}^2 = \iint_{(0,1)\times(0,1)} \frac{|v(x)-v(y)|^2}{|x-y|^{1+2s}}\,dx\,dy,
        \end{equation}
        and its localized version $\semin{v}{s}(E)$ on a Borel set \(E \subset (0,1)\) by
        \begin{equation}\label{def_Gagliardo_local}     
        \semin{v}{s}(E)^2 = \iint_{E\times E} \frac{|v(x)-v(y)|^2}{|x-y|^{1+2s}}\,dx\,dy.
        \end{equation}
            
		We start by dealing with the case $s_\e\to{\frac{1}{2}}^+$. Let $G_\e$ denote the unscaled functionals
		$$G_\e(u)=\frac1\e\int_{(0,1)}W(u)\,dt+\e^{2s_\e-1}\iint_{(0,1)\times(0,1)}\frac{\abs{u(x)-u(y)}^2}{\abs{x-y}^{1+2s_\e}}\,dx\,dy.$$
        
        We now give an argument leading to the computation of the correct scaling. For symmetry reasons, this is carried on for functionals defined on $(-1,1)$ and not on $(0,1)$. This allows to consider a sequence of functions 
		$u_\e$ which converges in measure to the function $$u_0=\begin{cases}
			1\quad\text{in }[0,1] \\
			-1\quad\text{in }[-1,0).
		\end{cases}$$
		We let $\eta\in(0,\frac14)$ and define 
        \[\sigma_\e=\abs{\{\abs{u_\e}\leq1-\eta\}},\quad C_\eta=\min_{\abs{z}\leq 1-\eta} W(z)>0,\]
		\[A_\e=\{u_\e>1-\eta\},\quad B_\e=\{u_\e<\eta-1\}.\]
		Then, we have the estimate:
		$$G_\e(u_\e)\geq \frac{\sigma_\e C_\eta }{\e} + 2\e^{2s_\e-1}\iint_{A_\e\times B_\e} \frac{\abs{u(x)-u(y)}^2}{\abs{x-y}^{1+2s_\e}}\,dx\,dy
        $$ 
		$$\geq\frac{\sigma_\e C_\eta }{\e} + 2(2-2\eta)^2 \e^{2s_\e-1}\iint_{A_\e\times B_\e} \frac{1}{\abs{x-y}^{1+2s_\e}}\,dx\,dy.$$
        Since the function $\psi(\abs{x-y})=\frac{1}{\abs{x-y}^{1+2s_\e}}$ is monotonically decreasing in $\abs{x-y}$, we obtain a lower bound by increasing the distance between $x$ and $y$ (see also Lemma 2 of \cite{ABS}). This leads to
		\[G_\e(u_\e)\geq\frac{\sigma_\e C_\eta }{\e} + 2(2-2\eta)^2 \e^{2s_\e-1} \int_{-1}^{-1+\abs{A_\e}} \int_{1-\abs{B_\e}}^{1} \frac{1}{\abs{x-y}^{1+2s_\e}}\,dx\,dy\]
		\[\begin{split}
		    =\frac{\sigma_\e C_\eta }{\e}+\frac{(2-2\eta)^2\e^{2s_\e-1}}{s_\e(2s_\e-1)}&\bigg(2^{1-2s_\e}-(2-\abs{B_\e})^{1-2s_\e}\\
            &-(2-\abs{A_\e})^{1-2s_\e}+(2-\abs{B_\e}-\abs{A_\e})^{1-2s_\e}\bigg),
		\end{split}\]
		which simplifies to
		$$\frac{\sigma_\e C_\eta }{\e} + C\e^{2s_\e-1}\frac{\sigma_\e^{1-2s_\e}-1}{2s_\e-1} + O(\e^{2s_\e-1})$$
        for some positive constant $C$ which does not depend on the other parameters.
        We now minimize the principal part with respect to $\sigma=\sigma_\e$. The minimum is attained for
		\begin{equation}\label{def_sigmaeps_minimo}
		    \sigma_\e =K \e\quad\text{with }K=\left(\frac{C}{C_\eta}\right)^\frac{1}{2s_\e},
		\end{equation}
		which leads to (possibly renaming $C$, which still does not depend on the other parameters)
		\begin{equation}\label{CalcoloScaling+}
		    G_\e(u_\e)\geq 
		C\left[ C_\eta^{1-\frac{1}{2s_\e}}  + \frac{\e^{2s_\e-1}}{2s_\e}\left(\frac{({K\e})^{1-2s_\e}-1}{2s_\e-1}\right)-1\right].
		\end{equation}
        Since $K^{1-2s_\e}\to1$ as $s_\e\to\frac12$, the leading term in the right-hand side is
        \[
        \frac{{\e}^{2s_\e-1}-1}{1-2s_\e}.
        \]
		This computation suggests the scaling factor
		\begin{equation}\label{Lambda+}
		    \lambda_+(\e)=\frac{1-2s_\e}{{\e}^{2s_\e-1}-1}.
		\end{equation}
        
        Note that, since we are interested in vanishing perturbations of the double-well functional, we ought to make sure that the scaling factor we have found behaves correctly (that is, we want the coefficient in front of the double well potential to diverge as $\e\to0$ and the one in front of the Gagliardo seminorm to tend to $0$).
        Let us check the validity of such conditions: in this case, the functional has the form
        $$F^{s_\e}_\e(u)=\frac{\lambda_+(\e)}{\e}\int_{(0,1)}W(u)\,dt+\lambda_+(\e)\e^{2s_\e-1}\semin{u}{s_\e}^2,$$
        thus, the conditions are
        $$\lambda_+(\e)\e^{2s_\e-1}\to0
        \quad\text{and}\quad
        \frac{\lambda_+(\e)}{\e}\to+\infty.$$
        Here, the first condition is equivalent to 
        $$0\leftarrow\frac{2s_\e-1}{\e^{1-2s_\e}-1}=\frac{(2s_\e-1)\abs{\log(\e)}}{e^{(2s_\e-1)\abs{\log(\e)}}-1}\cdot\frac{1}{\abs{\log(\e)}},$$
        which is satisfied, since the function $f(x)=\frac{e^x-1}{x}$ is bounded from below by a positive constant for all $x\in [0,+\infty)$.

    \begin{remark}\label{Cond1_also_s_small}
        The condition
        \[
        \lambda_+(\e)\e^{2s_\e-1}\to0
        \]
        is also satisfied if $s_\e<\frac12$. Indeed, if $\abs{\log(\e)}({2s_\e-1})$ is bounded the reasoning above still applies. 
        On the other hand, if $\abs{\log(\e)}({2s_\e-1})\to -\infty$, one has
        $$\frac{2s_\e-1}{e^{(2s_\e-1)\abs{\log(\e)}}-1}\to0.$$
        The reason why we separately study the cases $s_\e>\frac12$ and $s_\e<\frac12$ lies in the behaviour of the term in front of the double well potential, which may not diverge for $s<\frac{1}{2}$ with this scaling factor.
    \end{remark}
\noindent
    We now check the second condition, namely:
    $$\frac{\lambda_+(\e)}{\e}\to+\infty.$$
    This, after simplifications, leads to
    $$\frac{(2s_\e-1)\abs{\log(\e)}}{1-e^{(1-2s_\e)\abs{\log(\e)}}}\cdot\frac{1}{\e^{2s_\e}\abs{\log(\e)}}\to+\infty.$$
    Here, the first term is bounded from below whenever $\abs{\log(\e)}({2s_\e-1})$ is bounded, while $\frac{1}{\e^{2s_\e}\abs{\log(\e)}}\to+\infty$. On the other hand, if $\abs{\log(\e)}({2s_\e-1})\to+\infty$, both terms in the product tend to $+\infty$.

        If $s_\e\to{\frac{1}{2}}^-$, it suffices to repeat the same computation above, considering the unscaled functionals
        $$G_\e(u)=\frac1\e\int_{(0,1)}W(u)\,dt+\iint_{(0,1)\times(0,1)}\frac{\abs{u(x)-u(y)}^2}{\abs{x-y}^{1+2s_\e}}dx\,dy.$$
        In this case, the optimal value in the minimization procedure (performed as in \eqref{CalcoloScaling+}) is attained for 
		\begin{equation}
		    \sigma_\e =K \e^{\frac{1}{2s_\e}},
		\end{equation}
        with $K$ as in \eqref{def_sigmaeps_minimo},
        which leads to 
        \begin{equation}
            G_\e(u_\e)\geq 
        C\left[ C_\eta \e^{\frac{2s_\e-1}{2s_\e}}  + \frac{1}{2s_\e}
        \left(\frac{({K\e^\frac{1}{2s_\e}})^{1-2s_\e}-1}{2s_\e-1}\right)-1\right]
        \end{equation}
        for some positive constant $C$ which does not depend on the parameters. 
        We recall that the value $K$ satisfies $K^{1-2s_\e}\to1$ as $s_\e\to\frac12$. Hence, the leading term in the right-hand side is
        \[
        \frac{{\e}^{\frac{1-2s_\e}{2s_\e}}-1}{2s_\e-1}.
        \]
   This computation suggests the scaling factor
        \begin{equation}\label{Lambda-}
            \lambda_-(\e)=
        \frac{2s_\e-1}{{\e}^{\frac{1-2s_\e}{2s_\e}}-1}.
        \end{equation}

   \noindent In this case, the conditions become
    $$\lambda_-(\e)\to0
    \quad\text{and}\quad
    \frac{\lambda_-(\e)}{\e}\to+\infty.$$
    The first, being equivalent to 
    $$\frac{(2s_\e-1)\abs{\log(\e)}}{e^{(2s_\e-1)\abs{\log(\e)}}-1}\cdot\frac{1}{\abs{\log(\e)}}\to0,$$
    is satisfied by Remark \ref{Cond1_also_s_small}. On the other hand, the second condition leads to
    $$\frac{2s_\e-1}{\e({\e}^{\frac{1-2s_\e}{2s_\e}}-1)}\to+\infty,$$
    which is equivalent to
    $$\frac{2s_\e-1}{\e({\e}^{1-2s_\e}-1)}\to+\infty.$$
    Once again, we write it as
    $$+\infty\leftarrow\frac{2s_\e-1}{\e({\e}^{1-2s_\e}-1)}=\frac{(2s_\e-1)\abs{\log(\e)}}{e^{(2s_\e-1)\abs{\log(\e)}}-1}\cdot\frac{1}{\e\abs{\log(\e)}},$$
    where the first term in the product is bounded from below by a positive constant if $(2s_\e-1)\abs{\log(\e)}$ is bounded and tends to $+\infty$ if $(2s_\e-1)\abs{\log(\e)}\to-\infty$. The second term, on the other hand, always tends to $+\infty$, which implies that the second condition is satisfied for any regime of $s<\frac12$.
    
	\begin{remark}
		For $\e\in(0,1)$ fixed,
		$$\lim_{s\to{\frac{1}{2}}^+}\,\frac{1-2s}{\e^{2s-1}-1}\e^{2s-1} 
        =
        \lim_{s\to{\frac{1}{2}}^+}\,\frac{2s-1}{\e^{1-2s}-1}
        =
        \frac{1}{\abs{\log(\e)}}$$
        and
		$$\lim_{s\to\frac12^-}\frac{2s-1}{{\e}^{\frac{1-2s}{2s}}-1}
        =
        \frac{1}{\abs{\log(\e)}}.$$
		This allows us to recover the scaling factor for the critical regime $s=\frac12$.
	\end{remark}

	\section{The Gamma limit}\label{sec:gammalim}
	Let $(s_\e)_\e$ be a sequence converging to $\frac12$ as $\e\to0$ and consider the functionals
\begin{equation}\label{Functionals-Fe}
    F_\e(u)=
    \begin{cases}
        \displaystyle\frac{\lambda_+(\e)}{\e}\int_{(0,1)}W(u)\,dt+\lambda_+(\e)\e^{2s_\e-1}\semin{u}{s_\e}^2\quad \text{if } s_\e>\frac12\\\\
        \displaystyle\frac{1}{\abs{\log{\e}}\e}\int_{(0,1)}W(u)\,dt+\frac{1}{\abs{\log{\e}}}\semin{u}{s_\e}^2\quad \text{if } s_\e=\frac12\\   \\
        \displaystyle\frac{\lambda_-(\e)}{\e}\int_{(0,1)}W(u)\,dt+\lambda_-(\e)\semin{u}{s_\e}^2\quad \text{if } s_\e<\frac12\,,
    \end{cases}
\end{equation}
    with $\lambda_+(\e)$ and $\lambda_-(\e)$ being the scaling factors defined in \eqref{Lambda+} and \eqref{Lambda-} respectively and $\semin{v}{s_\e}$ being the Gagliardo seminorm defined in \eqref{def_Gagliardo}.

    In this section we prove that, as $\e\to0$, the functionals $F_\e$ $\Gamma$-converge to the functional $F^\frac12_0$ defined in \eqref{def_F0_12}. More precisely, we prove the following result.
	
	\begin{theo}\label{GammaLimite}
		Assume that $W$ is a double-well potential satisfying \eqref{HP_W_1}, \eqref{HP_W_2} and \eqref{HP_W_3}, let $\e>0$ and $s_\e\in(0,1)$ be such that $$\lim_{\e\to0} s_\e =\frac12\,.$$
		Then the family of functionals $(F_\e)_\e$ defined in \eqref{Functionals-Fe} $\Gamma$-converge, as $\e\to0^+$, to the functional 
		$$F_0=\begin{cases}
			m_\frac12\#S(u) \quad\text{if } u\in \BV\\
			+\infty\quad\text{otherwise,}\end{cases}$$
            with $m_\frac12=8$.
	\end{theo}

    We now state a technical lemma, which will be used multiple times throughout this paper, which sums up the minimization technique applied in the previous section. Its proof can be found in the appendix.
    
    \begin{lemma}\label{KeyLemma}
            Let $I\subset [0,1]$ be an interval, $s_\e\in(0,1)$, $s_\e\neq\frac12$ and $v\in H^{s_\e}(0,1)$. For any $\eta\in(0,\frac14)$, define 
            \[C_\eta=\min_{\abs{z}\leq 1-\eta} W(z)>0\]
            and
            \begin{equation}\label{Z-keylemma}
                Z=\left(\frac{8(1-\eta)^2\abs{I}^{-2s_\e}}{2s_\e C_\eta}\right)^{\frac{1}{2s_\e}}.
            \end{equation}
            Moreover, assume that there exists $\theta\in(0,\frac12)$ such that
            \begin{equation}\label{eq:keylemma}
                \abs{I\cap\{v\geq 1-\eta\}}\geq \theta\abs{I}\quad \text{and}\quad
                \abs{I\cap\{v\leq\eta-1\}}\geq \theta\abs{I}.
            \end{equation}
        Then, for $s_\e>\frac12$,
		\begin{equation}\begin{split}
		    F_\e(v;I)\geq \lambda_+(\e)
		\Biggr[ C_\eta \abs{I}Z + 8 &(1-\eta)^2\frac{\e^{2s_\e-1}\abs{I}^{2s_\e-1}}{2s_\e}\\ &\left(\frac{Z^{1-2s_\e}\e^{1-2s_\e}-1}{2s_\e-1}-2\frac{1-\theta}\theta\right)\Biggr],
		\end{split}
		\end{equation}
        and, for $s_\e<\frac{1}{2}$,
		\begin{equation}
        \begin{split}
		    F_\e(v;I)\geq \lambda_-(\e)
		\Biggr[C_\eta\abs{I}\e^{\frac{1}{2s}-1}Z+ 8 &(1-\eta)^2\frac{\abs{I}^{2s_\e-1}}{2s_\e}\\ &\left(\frac{Z^{1-2s_\e}\e^\frac{{1-2s_\e}}{2s_\e}-1}{2s_\e-1}-2\frac{1-\theta}\theta\right)\Biggr].
        \end{split}
		\end{equation}
        \end{lemma}
        \begin{remark}\label{Remark_keylemma}
        In both cases, for any $\theta$ and $\eta$ fixed, there exists a value $\overline \e > 0$ such that the right-hand side is bounded from below by a positive constant $C_0(\theta,\eta)$ for all $\e \in (0, \overline \e)$. Moreover, we observe that this constant tends to $-\infty$ as $\theta \to 0^+$.
        \end{remark}
	
Before proceeding to the proof of Theorem \ref{GammaLimite} we state and prove the corresponding compactness result. Its proof follows the technique employed in \cite{SVa-nonlocal-ext}.

Since the case $s_\e=\frac12$ has already been addressed in \cite{ABS}, we will only prove the theorem in the case $s_\e\neq\frac12$. Moreover, when the cases $s_\e>\frac12$ and $s_\e<\frac12$ can be handled using the same arguments, we will use the notation $\lambda_\pm(\e)$ to refer to either $\lambda_+(\e)$ or $\lambda_-(\e)$ without further specification on $s_\e$. 

Given any open subset $E\subset(0,1)$, we define the restricted energy
	$$F_\e(u;E)=
        \begin{cases}
            \displaystyle\frac{\lambda_+(\e)}{\e}\int_{E}W(u)\,dt+\lambda_+(\e)\e^{2s_\e-1}\semin{u}{{s_\e}}(E)^2\quad \text{if } s_\e>\frac12\\
            \\\displaystyle\frac{\lambda_-(\e)}{\e}\int_{E}W(u)\,dt+\lambda_-(\e)\semin{u}{{s_\e}}(E)^2\quad \text{if } s_\e<\frac12,
        \end{cases}
    $$
    where $\semin{u}{{s_\e}}(E)$ denotes the localized version of the Gagliardo seminorm, as in \eqref{def_Gagliardo_local}.

    We also introduce the notation $\mathcal{T}_k$ for the truncation function at height $\pm k$ (which will be used in the proof of the next lemma). More precisely, $\mathcal{T}_k$ is defined as follows:
        \begin{equation}\label{DefTroncate}
            \mathcal{T}_k(s)=\begin{cases}
            k\operatorname{sgn}(s) \quad\text{if } \abs{s}\geq k\\
            s\quad \text{otherwise.}
        \end{cases}
        \end{equation}
    
	\begin{lemma}\label{compactness1}
		Assume that $$\sup_{\e>0}F_\e(u_\e)\leq S<+\infty.$$
		Then there exists $u\in\BV$ such that, up to subsequences, $u_\e\to u$ in measure.
	\end{lemma}
    
	\begin{proof} 
    \textbf{Step 1:} we start by noting that we may suppose that the functions $u_\e$ are uniformly bounded. 
    
    Indeed, for any $k>1$ let $\mathcal{T}_k(u_\e)$ denote the truncation of $u_\e$ at height $\pm k$. Since the Gagliardo seminorm $\semin{u}{s_\e}$ decreases under truncation (cf \cite[Section 5.6]{Leoni-Fracsob}), we have
\begin{eqnarray*}
\sup_{\e>0}F_\e(\mathcal{T}_k(u_\e))&\leq& S+
\frac{\lambda_{\pm}(\e)}{\e}\int_{\{\abs{u_\e}>k\}} W(\mathcal{T}_k(u_\e))\,dt
\\
&\leq& S+
\frac{\lambda_{\pm}(\e)}{\e}\abs{\{\abs{u_\e}>k\}} (W(k)+W(-k))
\\
&\leq& S+
(W(k)+W(-k))
\frac1{C_k} \frac{\lambda_{\pm}(\e)}{\e}\int_{\{\abs{u_\e}>k\}} W(u_\e)\,dt
\\
&\leq&S'<+\infty.
\end{eqnarray*}
In this formula we have set $C_k=\min_{\abs{z}\geq k} W(z)$, which is strictly positive by the assumptions on the double-well potential.
This shows that the truncated functions satisfy the assumption of the lemma. 

We can then deduce that the claim of the lemma holds also for $u_\e$ once it holds for $\mathcal{T}_k(u_\e)$. Indeed, since we have 
\begin{equation}\label{Stima_SmallErrorTruncation}
        \int_{\{\abs{u_\e}>k\}} W(u_\e)\,dt\leq S\frac{\lambda_{\pm}(\e)}{\e}\to0 \quad\text{as }\e\to0,
    \end{equation}
    we obtain
    \begin{equation}\abs{\{u_\e \neq \mathcal{T}_k(u_\e)\}}=
        \abs{\{\abs{u_\e}>k\}}\leq \frac{1}{C_k}\int_{\{\abs{u_\e}>k\}} W(u_\e)\,dt\to0\quad \text{as }\e\to0;
    \end{equation}
 that is, the difference of $u_\e$ and $\mathcal{T}_k(u_\e)$ tends to $0$ in measure as $\e\to 0$.

    \textbf{Step 2:}
        we claim that the set $\{u_{\e}\}_\e$ is totally bounded in $L^1(0,1) $, that is, for any $\delta>0$ there exists a finite set $\mathcal A \subset L^1$ such that for any small $\e$ there exists $ v_{\e} \in \mathcal A$ with
        $$
        \left\|u_{\e}- v_{\e}\right\|_{L^1} \leq \delta.
        $$
        This implies that $u_\e$ converges, up to subsequences, to a cluster point of the family $\{v_\e\}_\e$.
        
		Let $\eta\in(0,\frac14)$ and $\rho>0$. We begin by partitioning the interval $[0,1]$ in $\frac{1}{\rho}$ (which we assume to be an integer) subintervals $I_i^\rho$ of length $\rho$. We will refer to the set of indexes in  such partition as $\mathcal{J}_\e$ and further classify these intervals as follows:
		
		$$\mathcal{J}_\e^{1} = \left\{i\in\mathcal{J}_\e \quad\colon\, \abs{\{{u_\e<1-\eta}\}\cap I_i^\rho}<\eta\rho \right\}$$

		$$\mathcal{J}_\e^{-1}= \left\{i\in\mathcal{J}_\e \quad\colon\, \abs{\{{u_\e>-1+\eta}\}\cap I_i^\rho}<\eta\rho \right\}$$
		
		$$\mathcal{J}_\e^0 =\mathcal{J}_\e\setminus(\mathcal J ^{1}_\e\cup \mathcal J ^{-1}_\e).$$

        \noindent For simplicity, we also call $$K_\e^1=\bigcup_{\mathcal J_\e^1} I_i^\rho\quad K_\e^{-1}=\bigcup_{\mathcal J_\e^{-1}} I_i^\rho\quad K_\e^{0}=\bigcup_{\mathcal J_\e^0} I_i^\rho.$$
        Then, we define
        \begin{equation}
            v_\e^\rho=\begin{cases}
                +1\quad\text{in } K^1_\e\\
                -1\quad\text{elsewhere.}
            \end{cases}
        \end{equation}
        Note that, for every fixed $\rho$, the family $\{v_\e^\rho\}_\e$ is finite.
        
		We now estimate from below the energy $F_\e(u;I_i^\rho)$ for $i\in\mathcal{J}_\e^0$. By subadditivity of $F_\e(\cdot,I)$ with respect to $I$, we have
        \begin{equation}\label{corr1}            
        S\geq \sum_{\mathcal{J}_\e^0}F_\e(u_\e;I_i^\rho).
        \end{equation}
        Moreover, by definition of $\mathcal{J}_\e^0$,
        \[\abs{\{{u_\e>1-\eta}\}\cap I_i^\rho}\geq\eta\rho\quad\text{and}\quad\abs{\{{u_\e<-1+\eta}\}\cap I_i^\rho}\geq\eta\rho.\]
        This implies, using Lemma \ref{KeyLemma} and \eqref{corr1} (and Remark \ref{Remark_keylemma}) that, for $\e$ small enough
        \begin{equation}
            S\geq \sum_{\mathcal{J}_\e^0}F_\e(u_\e;I_i^\rho)\geq
        \abs{\mathcal{J}_\e^0} C_0(\eta)
        \end{equation}
        for some positive constant $C_0$.
        It follows that, for $\rho$ small enough (depending on $\eta$), 
		\begin{equation}\label{Stima_K0}
			\abs{K_\e^{0}}=\sum_{\mathcal{J}_\e^0}\abs{I_i^\rho}\leq \frac{S}{C_0(\eta)}\rho\leq \frac{\delta}{8}.
		\end{equation}

    From this point onward, by the reasoning performed in Step 1, we may assume that $\abs{u_\e}\leq 1+\eta$ a.e. for every $\e>0$.
    
        We now estimate the $L^1$ distance between $u_\e$ and $v_\e$ outside of $K_\e^{0}$. We will do so by estimating such distance in the subsets of $K_\e^1$ (resp. $K_\e^{-1}$) where $u_\e$ is away from $\pm1$, when it is close to $-1$ and when it is close to $1$.
        First of all, we have
        \[
        \abs{\{\abs{u_\e}\leq 1-\eta \}}
        \leq 
        \frac{1}{C_\eta}\int_0^1 W(u_\e)\,dt\leq 
        \frac{\e}{\lambda_\pm(\e)}\frac1{C_\eta} S.
        \]
        Thus, since $\abs{u_\e-v_\e}\leq3$,
        \begin{equation}\label{stimainmezzo}
            \int_{\{\abs{u_\e}\leq 1-\eta \}}\abs{u_\e-v_\e}dt
            \leq 3\abs{\{\abs{u_\e}\leq 1-\eta \}}\leq \frac{\e}{\lambda_\pm(\e)}\frac{3S}{C_\eta}.
        \end{equation}
        Regarding the set where $u_\e$ is close to $-1$, we have
        \[
        \abs{K_\e^1\cap \{{u_\e}\leq \eta-1 \}}
        =
        \sum_{\mathcal{J}_\e^1}\abs{I^\rho_i\cap \{{u_\e}\leq \eta-1 \}}
        \leq
        \abs{\mathcal{J}_\e^1}\eta\rho=\eta\abs{K_\e^1},
        \]
        which implies, 
        \[\int_{K_\e^1\cap \{{u_\e}\leq \eta-1 \}}\abs{u_\e-v_\e}dt\leq 3\eta\abs{K_\e^1}.\]
        Moreover, we have
        \[
        \int_{K_\e^1\cap \{{u_\e}\geq 1-\eta \}}\abs{u_\e-v_\e}dt=
        \int_{K_\e^1\cap \{{u_\e}\geq 1-\eta \}}\abs{u_\e-1}dt\leq2
        \eta\abs{K_\e^1}.
        \]
        This and \eqref{stimainmezzo} lead to
        \begin{equation}\label{L1dist-vicino1}
                \int_{K_\e^1}\abs{u_\e-v_\e}dt\leq 5\eta+\frac{\e}{\lambda_\pm(\e)}\frac{3S}{C_\eta}.
        \end{equation}
        Performing the same computations for $K_\e^{-1}$ (up to exchanging the roles of the sets where $u_\e$ is close to $1$ and $-1$), we obtain
        \[
        \abs{K_\e^{-1}\cap \{{u_\e}\geq 1-\eta\}}
        \leq\eta\abs{K_\e^{-1}}
        \quad
        \text{and}
        \quad
        \int_{K_\e^{-1}\cap \{{u_\e}\leq \eta-1 \}}\abs{u_\e-v_\e}dt
        \leq
        2\eta\abs{K_\e^1},
        \]
        which, as before, implies
        \begin{equation}\label{L1dist-vicino-1}
                \int_{K_\e^{-1}}\abs{u_\e-v_\e}dt\leq 5\eta+\frac{\e}{\lambda_\pm(\e)}\frac{3S}{C_\eta}.
        \end{equation}
        Putting \eqref{Stima_K0}, \eqref{L1dist-vicino1} and \eqref{L1dist-vicino-1} together we obtain, for $\eta$, $\e$ and $\rho$ small enough,
        \[\begin{split}
            \int_{[0,1]}\abs{u_\e-v_\e}dt&\leq 
            \int_{K_\e^1}\abs{u_\e-v_\e}dt+
            \int_{K_\e^{-1}}\abs{u_\e-v_\e}dt+
            \int_{K_\e^{0}}\abs{u_\e-v_\e}dt\\
            & \leq2\left(5\eta+\frac{\e}{\lambda_\pm(\e)}\frac{3S}{C_\eta}\right) +3 \abs{K_\e^{0}}\leq \delta.
        \end{split}\]
    
	This proves the total boundedness of $\{u_\e\}_\e$. It follows that $u_\e$ converges (up to subsequences) to some function $u\in\elle1(0,1)$ which is a cluster point for the family $\{v_\e^\rho\}_{\e}$. Moreover, since $\abs{v_\e}=1$ for all $\e$, it follows that $u=2\rchi_E-1$ for some set $E\subset (0,1)$. Note that the choice of the limit $u$ is not affected by the height $k>1$ of the truncation we applied at the beginning of the proof.

    \textbf{Step 3:}
    we claim that $\#S(u)$ is finite. To prove it, define the functions
    $\psi_\rho$ to be equal to $1$ in $I_i^\rho$ if $\abs{E\cap I_i^\rho}\geq\frac{\rho}{2}$ and $-1$ otherwise. As $\psi_\rho\to u$ in measure as $\rho\to0^+$, it suffices to show that       
    $\#S(\psi_\rho)$ is equibounded. 
    Let $I_i^\rho$ be an interval such that $\psi_\rho$ has a jump point, denoted by $t_i$, either at the beginning or the end of $I_i^\rho$.  Consider the interval $\mathcal I_i=(t_i-\rho, t_i+\rho)$.
    
    Note that, by convergence in measure, 
    $$\abs{\{u_\e\geq 1-\eta\}\cap \mathcal I_i}=: a_\e \rho \geq \frac\rho2\quad\text{and}\quad \abs{\{u_\e\geq \eta-1\}\cap \mathcal I_i} =: b_\e \rho \geq \frac\rho2$$
    for $\e$ small enough.
    We can thus apply Lemma \ref{KeyLemma} (with $\theta=\frac14$) to each interval $\mathcal I_i$ to obtain
	\begin{equation}\begin{split}
		    F_\e(u_\e;\mathcal I_i)\geq \lambda_+(\e) 8 &(1-\eta)^2\frac{\e^{2s_\e-1}\abs{I}^{2s_\e-1}}{2s_\e}\left(\frac{Z^{1-2s_\e}\e^{1-2s_\e}-1}{2s_\e-1}-6\right),
		\end{split}
	\end{equation}
    with $Z$ given by \eqref{Z-keylemma},
    which tends to $8$ as $\e\to0^+$. In particular, it is definitely greater than a positive constant $C_0>0$ (see also Remark \ref{Remark_keylemma}). Analogously, for $s_\e<\frac12$, we obtain
	\begin{equation}
        \begin{split}
		    F_\e(u_\e;\mathcal I_i)\geq \lambda_-(\e) 8 &(1-\eta)^2\frac{\abs{I}^{2s_\e-1}}{2s_\e}\left(\frac{Z^{1-2s_\e}\e^\frac{{1-2s_\e}}{2s_\e}-1}{2s_\e-1}-6\right)\geq C_0>0.
        \end{split}
	\end{equation}

    Note that the intervals $\mathcal I_i$ may not be disjoint. However, by choosing them alternately, we can divide them in two family of disjoint intervals (indexed, say, by $J_1$ and $J_2$).
    In particular, we have that
    \[\begin{split}
        &S\geq F_\e(u_\e)\geq\sum_{i\in J_1}F_\e(u_\e,\mathcal I_i),\\
        &S\geq F_\e(u_\e)\geq\sum_{i\in J_2}F_\e(u_\e,\mathcal I_i).
    \end{split}
    \]
    Thus
    $$2S\geq 2F_\e(u_\e)\geq \sum_{t_i\in S(\psi_\rho)}F(u_\e,\mathcal I_i)\geq C_0 \#S(\psi_\rho),$$
    which concludes the proof.
	\end{proof}
	The next result deals with the lower bound.
	\begin{lemma}\label{liminf1}
		Let $u_\e\to u$ in measure. Then $u\in\BV$ and
		\begin{equation}
			\liminf_{\e\to0} F_\e(u_\e)\geq m_\frac12\#S(u).
		\end{equation}
	\end{lemma}
	
	\begin{proof}
		  Up to extracting a subsequence (not relabeled), we
          can assume that 
          \[\sup_{\e} F_\e(u_\e)<+\infty.\]
          This, by Lemma \ref{compactness1}, implies that $u\in\BV$.
		
		Note that for every $t_i\in S(u)$ there exists an interval $I_i=(t_i-\delta,t_i+\delta)$ such that $I_i\cap I_j=\emptyset$ if $i\neq j$. 
        Note that, by convergence in measure, the assumptions of Lemma \ref{KeyLemma} are satisfied with $\theta=\frac14$ for every $I_i$, provided that $\e$ is small enough.
        It follows that, for $s_\e>\frac12$,
		\begin{equation}\begin{split}
		    F_\e(u_\e)\geq&\sum_{i=1}^{\#S(u)} F_\e(u_\e;I_i)\\
            \geq & \sum_{i=1}^{\#S(u)}\lambda_+(\e)
		8 (1-\eta)^2\frac{\e^{2s_\e-1}\abs{I}^{2s_\e-1}}{2s_\e}\left(\frac{Z^{1-2s_\e}\e^{1-2s_\e}-1}{2s_\e-1}-6\right),
		\end{split}
		\end{equation}
        and, for $s_\e<\frac{1}{2}$,
		\begin{equation}
        \begin{split}
		    F_\e(u_\e)\geq&\sum_{i=1}^{\#S(u)} F_\e(u_\e;I_i)\\
            \geq & \sum_{i=1}^{\#S(u)}\lambda_-(\e)8 (1-\eta)^2\frac{\abs{I}^{2s_\e-1}}{2s_\e} \left(\frac{Z^{1-2s_\e}\e^\frac{{1-2s_\e}}{2s_\e}-1}{2s_\e-1}-6\right).
        \end{split}
		\end{equation}
        In both cases, taking the lower limit as $\e\to0^+$ we get
        \begin{equation}
            \liminf_{\e\to0^+} F_\e(u_\e)\geq 8 (1-\eta)^2\#S(u).
        \end{equation}
        This, letting $\eta\to0^+$, concludes the proof.
	\end{proof}
    The next lemma provides the upper bound inequality, which concludes the proof of the $\Gamma$-convergence in Theorem \ref{GammaLimite}.
    \begin{lemma}\label{limsup1}
		Let $u\in\BV$, then there exists a sequence $(u_\e)_\e$ such that $u_\e\in H^{s_\e}(0,1)$ such that $u_\e\to u$ in $\elle2(0,1)$ and 
		\begin{equation}
			\lim_{\e\to0} F_\e(u_\e)= m_\frac12\#S(u).
		\end{equation}
	\end{lemma}
	
	\begin{proof}
		For any jump point $t_i\in S(u)$, we consider the piecewise affine function defined as follows:
		$$u_\e(x)=\begin{cases}
			\mathcal{T}_1(\frac{x-t_i}{\e})u(t_i^+)\quad\text{if } \abs{x-t_i}<\e\\
			u(x)\quad\text{otherwise.}
		\end{cases}$$
        Here $\mathcal{T}_1$ is the truncation function at height $\pm1$ defined in \eqref{DefTroncate}.
        We also define the points $x_0=0$, $x_i=\frac{t_i+t_{i+1}}{2}$ for $1\leq i\leq\#S(u)-1$ and $x_{\#S(u)}=1$.
		
		First, note that $u_\e\to u$ in measure as $\e\to0$. This implies that
		$$\int_0^1 W(u_\e)\to \int_0^1 W(u)=0.$$
		
		To compute the Gagliardo seminorm of $u_\e$, we proceed as follows: 
		let us divide the square $(0,1)\times (0,1)$ in many rectangles, as in Figure \ref{Refpicture} (for reference, consider the solid line to be at the middle point between $t_1$ and $t_2$).
		
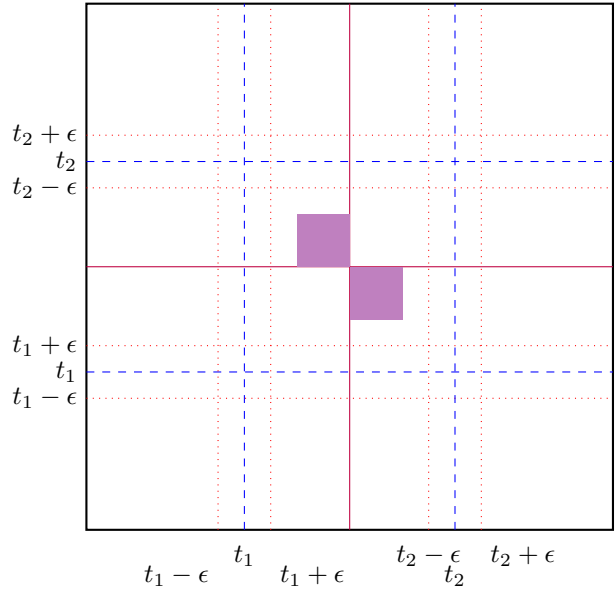
\begin{figure}[htbp]
		\begin{center}
			\begin{tikzpicture}[scale=7] 
				\def\tone{0.3}   
				\def\ttwo{0.7}   
				\def\eps{0.05}   
				
				\draw[thick] (0,0) rectangle (1,1);
				
				\foreach \t in {\tone, \ttwo} {
					\draw[dashed, color=blue] (\t,0) -- (\t,1);  
					\draw[dashed, color=blue] (0,\t) -- (1,\t);  
					
					\draw[dotted, color=red] (\t+\eps,0) -- (\t+\eps,1); 
					\draw[dotted, color=red] (\t-\eps,0) -- (\t-\eps,1); 
					\draw[dotted, color=red] (0,\t+\eps) -- (1,\t+\eps); 
					\draw[dotted, color=red] (0,\t-\eps) -- (1,\t-\eps); 
				}
				
				\node at (\tone, -0.05) {$t_1$};
				\node[below right] at (\tone+\eps, -0.05) {$t_1 + \e$}; 
				\node[below left] at (\tone-\eps, -0.05) {$t_1 - \e$}; 
				\node[below] at (\ttwo, -0.05) {$t_2$}; 
				\node[right] at (\ttwo+\eps, -0.05) {$t_2 + \e$}; 
				\node at (\ttwo-\eps, -0.05) {$t_2 - \e$}; 
				
				\node[left] at (0, \tone) {$t_1$};
				\node[left] at (0, \tone+\eps) {$t_1 + \e$};
				\node[left] at (0, \tone-\eps) {$t_1 - \e$};
				\node[left] at (0, \ttwo) {$t_2$};
				\node[left] at (0, \ttwo+\eps) {$t_2 + \e$};
				\node[left] at (0, \ttwo-\eps) {$t_2 - \e$};
				
				\draw[color=purple] (0.5,0) -- (0.5,1); 
				\draw[purple] (0,0.5) -- (1,0.5); 
				
				\fill[violet!50] (0.5-0.1,0.5) rectangle (0.5,0.5+0.1); 
				\fill[violet!50] (0.5,0.5-0.1) rectangle (0.5+0.1,0.5); 

			\end{tikzpicture}
		\end{center}
    \caption{Subdivision of the domain of integration}
    \label{Refpicture}
\end{figure}
		
		Our aim is to show that the contribution to Gagliardo seminorm $\semin{u_\e}{{s_\e}}$ given by the interaction of points at a distance greater than a given threshold can be bounded from above by a constant, so that the scaling factor (either $\lambda_-(\e)$ or $\lambda_+(\e)\e^{2s_\e-1}$) would make their contribution to $F_\e(u_\e)$ infinitesimal.
		
		First, note that on the purple squares (which can be taken with sides smaller than $\frac14\inf_{i}\abs{t_i-t_{i+1}}$) we have $u(x)-u(y)=0$. 
		Moreover, calling 
        \[Q_i=[x_i,x_{i+1}]\times [x_i,x_{i+1}],\] 
        we have
		\[\begin{split}
		    &\iint_{[0,1]\times [0,1]}\frac{\abs{u_\e(x)-u_\e(y)}^2}{\abs{x-y}^{1+2s_\e}}\,dx\,dy
		\\
        &=\sum_{i=0}^{\#S(u)-1}\iint_{Q_i}\frac{\abs{u_\e(x)-u_\e(y)}^2}{\abs{x-y}^{1+2s_\e}}\,dx\,dy+
		\iint_{[0,1]\times [0,1]\setminus \bigcup Q_i}\frac{\abs{u_\e(x)-u_\e(y)}^2}{\abs{x-y}^{1+2s_\e}}\,dx\,dy.
		\end{split}\]
		
		\noindent Working term by term, we have
		\[
        \iint_{Q_i}\frac{\abs{u_\e(x)-u_\e(y)}^2}{\abs{x-y}^{1+2s_\e}}\,dx\,dy= 2\int_{x_i}^{t_i-\e}\int_{t_i+\e}^{x_{i+1}}\frac{4}{\abs{x-y}^{1+2s_\e}}\,dx\,dy\,
        \]
		\[+\int_{t_i-\e}^{t_i+\e}\int_{t_i-\e}^{t_i+\e} \frac{\abs{u_\e(x)-u_\e(y)}^2}{\abs{x-y}^{1+2s_\e}}\,dx\,dy +
		4\int_{t_i+\e}^{x_{i+1}}\int_{t_i-\e}^{t_i+\e}\frac{\abs{u_\e(x)-1}^2}{\abs{x-y}^{1+2s_\e}}\,dx\,dy,\]
		which, after computations, leads to
		$$\iint_{Q_i}\frac{\abs{u_\e(x)-u_\e(y)}^2}{\abs{x-y}^{1+2s_\e}}\,dx\,dy=8 \frac{1}{2s}\frac{(2\e)^{1-2s_\e}-1}{2s_\e-1}+O(1).$$
		On the other hand,
		\[\iint_{[0,1]\times [0,1]\setminus \bigcup Q_i}\frac{\abs{u_\e(x)-u_\e(y)}^2}{\abs{x-y}^{1+2s_\e}}\,dx\,dy
        \leq \frac{C}{\inf_{i}\abs{t_i-t_{i+1}}^{1+2s_\e}}
        <+\infty.\]
        We now estimate the energy contribution from the double-well potential. In particular, we have:
        \[\int_0^1 W(u)\,dt=\sum_{i=1}^{\#S(u)}\int_{t_i-\e}^{t_i+\e} W\left(\frac {t-t_i}\e\right)\,dt\leq \#S(u) 2\e \max_{[-1,1]}W.\]
        
		\noindent Putting everything together we have, for $s_\varepsilon>\frac{1}{2}$,
        \[F_\varepsilon(u_\varepsilon)\leq C_1\lambda_+(\varepsilon) + 8\#S(u) \lambda_+(\varepsilon)\varepsilon^{2s_\varepsilon-1} \frac{1}{2s}\left(\frac{(2\varepsilon)^{1-2s_\varepsilon}-1}{2s_\varepsilon-1}+C_2\right),\]
        and, for $s_\varepsilon<\frac12$,
		$$F_\varepsilon(u_\varepsilon)=F_\varepsilon(u_\varepsilon)\leq C_1\lambda_-(\varepsilon) + 8\#S(u) \lambda_-(\varepsilon) \frac{1}{2s}\left(\frac{(2\varepsilon)^{1-2s_\varepsilon}-1}{2s_\varepsilon-1}+C_2\right).$$
        Either way, we proved that
        $$\limsup_{\varepsilon\to0} F_\varepsilon(u_\varepsilon) \leq m_\frac12\#S(u) $$
		which concludes the proof.
	\end{proof}
    
    \begin{remark}(Separation of scales)
        Assume that $$(2s_\e-1)\abs{\log{\e}}\to C\in\R.$$ Then
        \[
        \frac{1-2s_\e}{\varepsilon^{2s_\e-1}-1}\varepsilon^{2s_\e-1}
        \approx\frac{C }{e^{C}-1}\frac{1}{\abs{\log\e}}
        \]
        and
        \[
        \frac{2s_\e-1}{{\varepsilon}^{\frac{1-2s_\e}{2s_\e}}-1}
        \approx\frac{C }{e^{C}-1}\frac{1}{\abs{\log\e}}.
        \]
        In the limit case $C=0$, this computation highlights that in the regime 
        $$
        \abs{2s_\e-1}<<\frac1{|\log\e|},
        $$
        we have the {\em separation of scales} effect discussed in the introduction. Namely, the $\Gamma$-limit of $\lambda(\varepsilon,s) F_\varepsilon^{s}$ is the same as the one obtained first letting $s\to \frac12$ with $\e>0$ fixed, which gives  $\frac1{|\log\e|} F_\varepsilon^{1/2}$, and then letting $\e\to 0$. 
    \end{remark}
    
	\section{On the behaviour of optimal profiles}\label{sec:ms}
		In this section we briefly discuss the properties of the optimal values
		\begin{equation}\label{ms:def}
		    m_s=\inf \left\{\int_{\R} W(u) dt + \iint_{\R\times\R}\frac{\abs{u(x)-u(y)}^2}{\abs{x-y}^{1+2s}}\,dx\,dy\mid u\in H^s(\R),u(\pm\infty)=\pm1\right\}
		\end{equation}
		as $s\to{\frac{1}{2}}^+$.
        In particular, we prove that the value $m_\frac12$ is the solution of a (suitably rescaled) optimal profile problem and that
        \[
        \lim_{s\to\frac12^+}(2s-1)m_s=m_\frac12\,.
        \]
        Both properties are strongly related to the fact that, while the $\Gamma$-limit of $F^s_\varepsilon$ with $s>\frac{1}{2}$ involves the optimal constant $m_s$, the case $s=\frac12$ requires an additional scaling argument.

        \noindent For any $T\geq0$ and $s\geq\frac12$ consider the family $\mathcal H_T^s$ of real valued functions $v$ defined on $\R$ such that $v\in H^{s}(-T,T)$, $v(\pm T)=\pm1$, $v(t)=1$ if $t>T$ and $v(t)=-1$ if $t<-T$. 
        It was shown in \cite{PVi-GammaLim-s01} that, for any $s\in\left(\frac12,1\right)$, the optimal values $m_s$ satisfy
        \begin{equation}\label{eq:opt_s}
            m_{s}=\lim_{T\to+\infty}\inf \left\{\int_{-T}^T W(u) dt + \int_{-T}^T\int_{-T}^T\frac{\abs{u(x)-u(y)}^2}{\abs{x-y}^{{1+2s}}}\,dx\,dy\mid v\in\mathcal H_T^s\right\}.
        \end{equation}
        The following lemma gives a similar result for $s=\frac12$.
        \begin{lemma}
        The quantity \( m_{\frac{1}{2}} \) satisfies the following property:
            \begin{equation}\label{eq:opt12}
                m_{\frac12}=\lim_{T\to+\infty}\frac{1}{\log(2T)}\inf \left\{\int_{-T}^T W(u) dt + \int_{-T}^T\int_{-T}^T\frac{\abs{u(x)-u(y)}^2}{\abs{x-y}^{2}}\,dx\,dy\mid v\in\mathcal H_T^\frac12\right\}.
            \end{equation}
        \end{lemma}
        \begin{proof}Let 
            \[\tilde m_\frac12^T=\inf \left\{\int_{-T}^T W(u) dt + \int_{-T}^T\int_{-T}^T\frac{\abs{u(x)-u(y)}^2}{\abs{x-y}^{2}}\,dx\,dy\mid v\in\mathcal H_T^\frac12\right\}.\]
            For any function $v\in\mathcal{H}_T^s$, we fix $\eta\in(0,\frac14)$ and consider \[A=\{v>1-\eta\}\cap[-T,T],\quad B=\{v<\eta-1\}\cap[-T,T],\] 
            \[a=\frac{\abs{A}}{2T},\quad b=\frac{\abs{B}}{2T},\quad C_\eta=\min_{\abs{z}\leq 1-\eta}W(z).\]
            Following the techniques used in Section \ref{sec:scaling}, we have:
            \[\int_{-T}^T W(v) dt\geq C_\eta2T (1-a-b)\]
            and
            \[ \int_{-T}^T\int_{-T}^T\frac{\abs{v(x)-v(y)}^2}{\abs{x-y}^{2}}\,dx\,dy\geq2\iint_{A\times B}\frac{\abs{v(x)-v(y)}^2}{\abs{x-y}^{2}}\,dx\,dy\]
            \[\geq 8(1-\eta)^2\iint_{A\times B}\frac1{\abs{x-y}^{2}}\,dx\,dy
            \geq 8(1-\eta)^2\int_{-T}^{-T+\abs{B}}\int_{T-\abs{A}}^T\frac1{\abs{x-y}^{2}}\,dx\,dy\]
            \[\begin{split}
                =8(1-\eta)^2 [\log(2T-\abs{A})+&\log(2T-\abs{B})\\
                &\quad-\log(2T)-\log(2T-\abs{A}-\abs{B}) ]
            \end{split}\]
            \[=8(1-\eta)^2\left[\log(1-a)+\log(1-b)-\log(1-a-b) \right].\]
            Hence
            \[\int_{-T}^T W(v) dt +\int_{-T}^T\int_{-T}^T\frac{\abs{v(x)-v(y)}^2}{\abs{x-y}^{2}}\,dx\,dy\]
            \[\begin{split}
            \geq
            C_\eta2T (1-a-b) + 8(1-\eta)&^2[\log(1-a)+\\ &\quad \log(1-b)-\log(1-a-b) ]
            \end{split}\]
            \[
            \begin{split}
                =
                8(1-\eta)^2&
                \bigg\{\frac{C_\eta}{8(1-\eta)^2}2T (1-a-b)\\ 
                &+ \left[\log(1-a)+\log(1-b)-\log(1-a-b) \right]\bigg\}.
            \end{split}
            \]
            This, using the inequality
            \[-\log x+Mx\geq \log M\]
            with 
            \[
            x=(1-a-b)\quad\text{and}\quad M=\frac{C_\eta}{8(1-\eta)^2}2T,
            \]
            we obtain 
            \[
            \begin{split} 
                \int_{-T}^T W(v) dt +&\int_{-T}^T\int_{-T}^T\frac{\abs{v(x)-v(y)}^2}{\abs{x-y}^{2}}\,dx\,dy\\
                \geq 
                &8(1-\eta)^2\left\{\log\left(\frac{C_\eta}{8(1-\eta)^2}2T\right)+ \log(1-a)+\log(1-b)\right\}.
            \end{split}
            \]
            In conclusion, we have
            \begin{gather*}
            \frac{1}{\log(2T)}\left(\int_{-T}^T W(v) dt + \int_{-T}^T\int_{-T}^T\frac{\abs{v(x)-v(y)}^2}{\abs{x-y}^{2}}\,dx\,dy\right)\\
            \geq \frac{\tilde C(\eta)}{\log{(2T)}}+8(1-\eta)^2.
            \end{gather*}
            Hence, taking the limit as $T\to+\infty$ and then letting $\eta\to0^+$, we obtain
            \[\lim_{T_\to+\infty}\tilde m^T_\frac12\geq8.\]
            For the opposite inequality, we consider the test function $u(t)=\mathcal{T}_1(t)$. Indeed, for the potential term, we have
            \[\int_{-T}^T W(u)\,dt=\int_{-1}^1 W(t)\,dt.\]
            On the other hand, regarding the Gagliardo seminorm
            \[\int_{-T}^T\int_{-T}^T\frac{\abs{u(x)-u(y)}^2}{\abs{x-y}^{2}}\,dx\,dy,\]
            we split the integral as in the proof of Lemma \ref{limsup1}, obtaining
            \[
            \int_{-1}^1\int_{-1}^1\frac{\abs{x-y}^2}{\abs{x-y}^{2}}\,dx\,dy+4\int_{1}^T\int_{-1}^1 \frac{\abs{x-1}^2}{\abs{x-y}^{2}}\,dx\,dy+2\int_{1}^T\int_{1}^T \frac{4}{(x+y)^{2}}\,dx\,dy
            \]
            which is equal to
            \[
            C+8[2\log(1+T)-\log(2T)-\log(2)].
            \]
            Hence
            \[
            \frac{1}{\log(2T)}\left(\int_{-T}^T W(u) dt + \int_{-T}^T\int_{-T}^T\frac{\abs{u(x)-u(y)}^2}{\abs{x-y}^{2}}\,dx\,dy\right)= 8+O\left(\frac{1}{\log(2T)}\right).
            \]
            Taking the limit as $T\to+\infty$ concludes the proof.
        \end{proof}
        
		\begin{lemma}\label{Limite_ms}
			Given $m_s$ as in \eqref{ms:def}, one has
			\begin{equation}\label{ms_esplode}
				\lim_{s\to{\frac12}^+}(2s-1) m_s= m_{\frac12}.
			\end{equation}
		\end{lemma}
		\begin{proof}
			Let $v_s\in H^{\frac12}(\R)$ and $T>0$ such that $v_s(\pm T)=\pm1$ and $v_s$ is constant outside of $[-T,T]$.
            For any $\eta\in(0,\frac14)$, define: $$\delta=\delta_s(\eta)=\abs{\{\abs{v_s}\leq 1-\eta\}},$$
            $$A_\eta=\{v_s\geq 1-\eta\}\quad B_\eta=\{v_s\leq \eta-1\}.$$
            The energy $$F^s_1(v_s,\R)=\int_{\R} W(v_s) dt + \semin{u}{s}(\R)^2$$ satisfies:
            $$F^s_1(v_s,\R)\geq C_\eta \delta+2\iint_{A_\eta\times B_\eta} \frac{\abs{v_s(x)-v_s(y)}^2}{\abs{x-y}^{1+2s}}\,dx\,dy$$
            $$\geq C_\eta \delta+8(1-\eta)^2\iint_{A_\eta\times B_\eta} \frac{1}{\abs{x-y}^{1+2s}}\,dx\,dy.$$
            Since $\abs{\R\setminus(A_\eta\cup B_\eta)}=\delta$, translation invariance implies
            \begin{gather*}
                F^s_1(v_s,\R)\geq C_\eta \delta+8(1-\eta)^2\int_{\frac\delta2}^{+\infty} \int_{\frac\delta2}^{+\infty}\frac{1}{({x+y})^{1+2s}}\,dx\,dy\\
                \geq 
                C_\eta \delta+8(1-\eta)^2\frac{\delta^{1-2s}}{2s(2s-1)}.
            \end{gather*}
            Taking the infimum over all possible values of $\delta$, which is attained for \hbox{$\delta=\left(\frac{8(1-\eta)^2}{2sC_\eta}\right)^\frac{1}{2s}$}, yields
            $$F^s_1(v_s,\R)\geq C_\eta \left(\frac{8(1-\eta)^2}{2sC_\eta}\right)^\frac{1}{2s}+ \frac{1}{2s}\left(\frac{8(1-\eta)^2}{2sC_\eta}\right)^\frac{1-2s}{2s}(1-\eta)^2\frac{8}{2s-1}.$$
            Taking the infimum over all choices of $v_s$ and then multiplying both sides by $2s-1$ leads to
            $$(2s-1)m_s\geq (2s-1)C_\eta \left(\frac{8(1-\eta)^2}{2sC_\eta}\right)^\frac{1}{2s}+ \frac{1}{2s}\left(\frac{8(1-\eta)^2}{2sC_\eta}\right)^\frac{1-2s}{2s}(1-\eta)^28.$$
            We now take the limit as $s\to\frac12^+$ and then as $\eta\to0^+$ to conclude:
            $$\lim_{s\to{\frac12}^+}(2s-1) m_s\geq m_{\frac12}.$$

            \noindent We now prove the converse inequality: let $u(x)=\mathcal{T}_1(x)$. Then
            $$F_1^s(u,\R)= \int_1^1 W(x)\,dt+8\int_1^{+\infty}\int_1^{+\infty}\frac{1}{(x+y)^{1+2s}}\,dx\,dy$$
            $$+\iint_{[-1,1]^2}\abs{x-y}^{1-2s}\,dx\,dy+\int_{-1}^1\abs{x-1}^2\int_1^{+\infty}\frac{1}{\abs{x-y}^{1+2s}}dydx.$$
            Since, as $s\to\frac12^+$, one has
            \[
            \begin{split}
                \int_1^1 W(x)\,dx+&\iint_{[-1,1]^2}\abs{x-y}^{1-2s}\,dx\,dy\\
                +& \int_{-1}^1\abs{x-1}^2\int_1^{+\infty}\frac{1}{\abs{x-y}^{1+2s}}dydx=O(1),
            \end{split}
            \]
            the leading term is 
            $$8\int_1^{+\infty}\int_1^{+\infty}\frac{1}{(x+y)^{1+2s}}\,dx\,dy=8\frac{2^{\frac{1}{2s}}}{2s(2s-1)}.$$
            Hence we have
            $$m_s\leq F_1^s(u,\R)=8\frac{2^{\frac{1}{2s}}}{2s(2s-1)}+O(1),$$
            which implies
            $$\lim_{s\to{\frac12}^+}(2s-1) m_s\leq m_{\frac12}.$$
            This concludes the proof.
		\end{proof}

\section{Continuity and regular points}\label{sec:gammaexp}

    Consider the scaling factor, which is continuous in both variables
    \[\lambda(\varepsilon,s)=
        \begin{cases}
             \frac1{\abs{\log{\varepsilon}}}\quad s=\frac12\\
             \frac{1-2s}{\varepsilon^{(2s-1)}-1}\quad s\in(\frac12,1)
        \end{cases}
    \]
    and the family of functionals
    
    \[\mathcal{F}_\varepsilon^s(u)=\lambda(\varepsilon,s)\left(\frac1{\varepsilon}\int W(u)\,dt+\varepsilon^{2s-1}\semin{u}{s}^2\right)\quad s\in\left[\frac12,1\right).\]

    \noindent We can reinterpret our previous results through the lens of regular values introduced in \cite{BT-GammaExp} (see also \cite{Braides-Sigalotti} for a similar result in the context of periodically-perforated media). Namely, for any $s_0\in\left[\frac12,1\right)$ and any pair of sequences $(s_j,\varepsilon_j)\to(s_0,0)$, $(s_j^\prime,\varepsilon_j^\prime)\to(s_0,0)$, we have
    \[\Gamma\hbox{-}\lim_{j\to+\infty} \mathcal{F}_{\varepsilon_j}^{s_j} = 
    \Gamma\hbox{-}\lim_{j\to+\infty} \mathcal{F}_{\varepsilon_j^\prime}^{s_j^\prime}.
    \]
    Moreover, we point out that 
    \[\Gamma\hbox{-}\lim_{\varepsilon\to0} \mathcal{F}_\varepsilon^s(u) =\begin{cases}
         m_\frac12 F_0(u)&\hbox{ if }  s=\frac12\\
         (2s-1)m_s F_0(u)&\hbox{ if }  s\in\left(\frac12,1\right),
    \end{cases}\]
where
\[
F_0 (u)=
\begin{cases}
    \#S(u) & \text{if } u \in \BV, \\
    +\infty & \text{otherwise},
\end{cases}
\]
and $\#S(u)$ denotes the number of jump points of $u$ in $\BV$.
    Notably, the regularity of the point $s=\frac12$ and the continuity of the $\Gamma$-limits with respect to $s\geq\frac12$ only occur thanks to the presence of the scaling factor $\lambda(\varepsilon,s)$. 
This highlights another separation of scales. Specifically, in the regime
\[\abs{2s_\e-1}>>\frac1{|\log\e|},\]
the  $\Gamma$-limit of $\lambda(\varepsilon,s) F_\varepsilon^{s}$ coincides with the one obtained by first taking $\e\to 0$ with $s> \frac12$ fixed, yielding  $(2s-1){m_s}\#S(u)$, and then letting $s\to \frac12$. 
    
In the terminology of $\Gamma$-expansions (cf. \cite{BT-GammaExp}) we can state that 
$\int W(u)\,dt+\varepsilon^{2s}\semin{u}{s}^2$ is uniformly equivalent to 
    \[\begin{cases}
        \displaystyle\e|\log\e| m_\frac12\#S(u)\quad\hbox{ if } s=\frac12\\
        \\
        \displaystyle\e\big(1-\e^{2s-1}\big)m_s\#S(u)\quad\hbox{ if } s\in\left(\frac12,1\right)
    \end{cases}\]
for $s$ varying in compact sets of $[\frac12,1)$.

\section*{Appendix}

\begin{proof}[Proof of Lemma \ref{KeyLemma}]
    Let us define the sets
            \begin{equation}
                A={I\cap\{v\geq 1-\eta\}}\quad \text{and}\quad
                B={I\cap\{v\leq\eta-1\}}
            \end{equation}
    here, we omit the dependence on $\eta$ and $\e$ to simplify the notation. For the same reason, we will write $s$ instead of $s_\e$, as the dependence of $s_\varepsilon$ on $\varepsilon$ will not be used in this proof.
    Moreover, we define
    \begin{equation}
        \alpha=\frac{\abs{A}}{\abs{I}}\quad \text{and}\quad\beta=\frac{\abs{B}}{\abs{I}}.
    \end{equation}
    Note that
    \begin{equation}\label{alphabeta}
        \theta<\alpha,\beta<1-\theta.
    \end{equation}

    In the next steps, we compute a lower bound for the Gagliardo seminorm.
    \begin{equation}
        \semin{v}{s}(I)^2
        \geq 
        2\int_A\int_B\frac{\abs{u(x)-u(y)}^2}{\abs{x-y}^{1+2s}}\,dx\,dy
        \geq 
        2(2-\eta)^2\int_A\int_B\frac{1}{\abs{x-y}^{1+2s}}\,dx\,dy.
    \end{equation}
    Using the notation $I=(a,b)$ we get, by monotonicity of the function $\psi(\abs{x-y})=\frac{1}{\abs{x-y}^{1+2s_\e}}$ in $\abs{x-y}$,
    \begin{equation}
        \semin{v}{s}(I)^2
        \geq
        2(2-\eta)^2 \int_a^{a+\abs{A}}\int_{b-\abs{B}}^b\frac{1}{\abs{x-y}^{1+2s}}\,dx\,dy.
    \end{equation}
    Here, the right-hand side is equal to
    \begin{equation}
         \frac{2(2-\eta)^2}{2s(2s-1)}\Biggr((b-\abs{B}-a-\abs{A})^{1-2s}-(b-a-\abs{A})^{1-2s}-(b-\abs{B}-a)^{1-2s}+(b-a)^{1-2s}\Biggr),
    \end{equation}
    which is
    \begin{equation}
         \frac{2(2-\eta)^2\abs{I}^{1-2s}}{2s(2s-1)}\Biggr((1-\beta-\alpha)^{1-2s}-(1-\alpha)^{1-2s}-(1-\beta)^{1-2s}+1\Biggr).
    \end{equation}
    Now we let $\sigma=1-\beta-\alpha$ and rearrange the terms to get
    \begin{equation}\label{App2}
         \frac{2(2-\eta)^2\abs{I}^{1-2s}}{2s}
         \Biggr(
         \frac{\sigma^{1-2s}-1}{(2s-1)}
         -\frac{(1-\alpha)^{1-2s}-1}{(2s-1)}
         -\frac{(1-\beta)^{1-2s}-1}{(2s-1)}\Biggr).
    \end{equation}
    Note that the function
    \[t\longmapsto -\frac{(1-t)^{1-2s}-1}{(2s-1)}\]
    is monotonically decreasing, which implies that
    \[\frac{\theta^{1-2s}-1}{(2s-1)}\geq\frac{(1-\alpha)^{1-2s}-1}{(2s-1)}.\]
    Hence, by \eqref{alphabeta}, we have
    \[-\frac{(1-\alpha)^{1-2s}-1}{(2s-1)}\geq-\frac{\theta^{1-2s}-1}{(2s-1)}.\]
    Recall that, for $\theta\in(0,1)$, the function
    \[s\longmapsto -\frac{\theta^{1-2s}-1}{(2s-1)}\]
    is monotonically decreasing in $(0,1)\setminus\{\frac{1}{2}\}$, which implies
    \[-\frac{(1-\alpha)^{1-2s}-1}{(2s-1)}\geq -\frac{\theta^{1-2s}-1}{(2s-1)} \geq-\frac{1-\theta}{\theta}.\]
    Since the same holds true if we substitute $\alpha$ with $\beta$, by \eqref{App2} we get
    \begin{equation}\label{App3}
        \semin{v}{s}(I)^2
        \geq
         \frac{2(2-\eta)^2\abs{I}^{1-2s}}{2s}
         \Biggr(
         \frac{\sigma^{1-2s}-1}{(2s-1)}-2\frac{1-\theta}{\theta}\Biggr).
    \end{equation}
    
    We now restrict ourselves to the case $s>\frac12$, the other being identical. By \eqref{App3}, we have
    \begin{equation}
        F_\e(v;I)\geq \lambda_+(\e)\Biggr[\frac{1}{\e}\sigma \abs{I}C_\eta + 
         \frac{2(2-\eta)^2\abs{I}^{1-2s}}{2s}
         \Biggr(
         \frac{\sigma^{1-2s}-1}{(2s-1)}-2\frac{1-\theta}{\theta}\Biggr)\Biggr].
    \end{equation}
    We now minimize the right-hand side with respect to $\sigma$, which leads to
    \[
    \frac{1}{\e}\sigma \abs{I}C_\eta =
         \frac{2(2-\eta)^2\abs{I}^{1-2s}}{2s}\sigma^{-2s},
    \]
    that is,
    \[\sigma=Z\e,\quad\text{with}\quad Z=\left(\frac{2(2-\eta)^2\abs{I}^{-2s}}{2sC_\eta}\right)^\frac{1}{2s}.\]
    Choosing such value for $\sigma$, we get \eqref{eq:keylemma}.
\end{proof}

\section*{Acknowledgments}
I would like to thank Andrea Braides for his supervision and insightful guidance. I also wish to thank the anonymous referee for their valuable comments and suggestions, which helped improve the clarity of this work.\\
I am a member of the GNAMPA group of INdAM. I have no conflicts of interest.

\end{document}